\documentclass[12pt,leqno,twoside]{article}
\usepackage{amssymb}
\usepackage{amsmath}
\usepackage{amsthm}
\usepackage{t1enc}
\usepackage[cp1250]{inputenc}
\usepackage{a4,indentfirst,latexsym}
\usepackage{graphics}
\usepackage{mathrsfs}
\usepackage{cite,enumitem,graphicx}
\usepackage[colorlinks=true,urlcolor=black,
citecolor=black,linkcolor=black,linktocpage,pdfpagelabels,
bookmarksnumbered,bookmarksopen]{hyperref}
\usepackage[colorinlistoftodos]{todonotes}

\input xy
\xyoption{all}

\newtheorem{Th}{Theorem}[section]
\newtheorem{Prop}[Th]{Proposition}
\newtheorem{Lem}[Th]{Lemma}

\newtheorem{Rem}[Th]{Remark}

\newtheorem{Que}[Th]{Question}

\newenvironment{altproof}[1]
{\noindent
{\em Proof of {#1}}.}
{\nopagebreak\mbox{}\hfill $\Box$\par\addvspace{0.5cm}}

   \newcommand{\vp}{\varphi}
   
   \newcommand{\eps}{\varepsilon}

   \def\div{\mathop{\mathrm{div}\,}}

     \def\span{\mathrm{span}}

   \def\id{\mathrm{id}}
    \def\O{\mathrm{O}}

   \def\N{\mathbb{N}}
   \def\R{\mathbb{R}}

   \def\curl{\mathrm{curl}}
   \def\dim{\mathrm{dim}}

   \def\cl{\mathrm{cl\,}}

   \def\U{\mathcal{U}} 
   
   \def\V{\mathcal{V}}
   
   \def\J{\mathcal{J}}

   \def\W{\mathcal{W}}
   
   \def\C{\mathbb{C}}

\newcommand{\cB}{{\mathcal B}}
\newcommand{\cC}{{\mathcal C}}
\newcommand{\cD}{{\mathcal D}}
\newcommand{\cE}{{\mathcal E}}

\newcommand{\cH}{{\mathcal H}}

\newcommand{\cJ}{{\mathcal J}}

\newcommand{\cM}{{\mathcal M}}
\newcommand{\cN}{{\mathcal N}}

\newcommand{\cP}{{\mathcal P}}

\newcommand{\cT}{{\mathcal T}}

\newcommand{\cV}{{\mathcal V}}
\newcommand{\cW}{{\mathcal W}}

\newcommand{\fJ}{{\mathfrak J}}

\renewcommand{\dim}{{\rm dim}\,}

\newcommand{\al}{\alpha}
\newcommand{\be}{\beta}
\newcommand{\ga}{\gamma}
\newcommand{\de}{\delta}
\newcommand{\la}{\lambda}
\newcommand{\De}{\Delta}
\newcommand{\Ga}{\Gamma}
\newcommand{\Om}{\Omega}

\def\curlop{\nabla\times}
\newcommand{\weakto}{\rightharpoonup}
\newcommand{\pa}{\partial}
\def\id{\mathrm{id}}

\newcommand{\tX}{\widetilde{X}}
\newcommand{\tu}{\widetilde{u}}
\newcommand{\tv}{\widetilde{v}}
\newcommand{\tcV}{\widetilde{\cV}}

\newcommand{\wh}{\widehat}

\newcommand{\cTto}{\stackrel{\cT}{\longrightarrow}}

\numberwithin{equation}{section}

\begin{document}
\title{Ground and bound state solutions of semilinear time-harmonic Maxwell equations in a bounded domain}
\author{Thomas Bartsch \and Jaros\l aw Mederski\footnote{The study was supported by research fellowship within project "Enhancing Educational Potential of Nicolaus Copernicus University in the Disciplines of Mathematical and
Natural Sciences" (project no. POKL.04.01.01-00-081/10)}}
\date{}
\maketitle

\begin{abstract}
We find solutions $E:\Om\to\R^3$ of the problem
\[
\left\{
\begin{aligned}
&\curlop(\curlop E) + \la E = \pa_E F(x,E) &&\quad \text{in }\Om\\
&\nu\times E = 0 &&\quad \text{on }\pa\Om
\end{aligned}
\right.
\]
on a simply connected, smooth, bounded domain $\Om\subset\R^3$ with connected boundary and exterior normal $\nu:\pa\Om\to\R^3$. Here $\curlop$ denotes the curl operator in $\R^3$, 
the nonlinearity $F:\Om\times\R^3\to\R$ is superquadratic and subcritical in $E$. The model nonlinearity is of the form $F(x,E)=\Ga(x)|E|^p$ for $\Ga\in L^\infty(\Om)$ positive, some $2<p<6$. It need not be radial nor even in the $E$-variable. The problem comes from the time-harmonic Maxwell equations, the boundary conditions are those for $\Om$ surrounded by a perfect conductor.
\end{abstract}



\section{Introduction}\label{sec:intor}

The paper deals with the Maxwell equations
\begin{equation}\label{eq:Maxwell}
\left\{
\begin{aligned}
    &\curlop \cH = \cJ+ \pa_t \cD \quad\hbox{(Ampere's law)}\\
    &\div(\cD)=\rho\\
    &\pa_t \cB + \curlop \cE=0 \quad\hbox{(Faraday's law)}\\
    &\div(\cB)=0,
\end{aligned}
\right.
\end{equation}
in a domain $\Om\subset\R^3$. Here $\cE,\cB,\cD,\cH:\Om\times\R\to \C^3$
correspond to the electric field, magnetic field, electric displacement field and magnetic induction, respectively. $\cJ$ is the electric current intensity and $\rho$ the electric charge density. Let $\cP,\cM:\Om\times\R\to \C^3$ denote the polarization field and the magnetization field respectively, and let $\eps,\mu:\Om\to\R$ be the permittivity and the permeability of the material. Then we consider the constitutive relations
\begin{equation}\label{eq:relations}
\cD=\eps \cE+\cP \quad\text{and}\quad \cH=\frac1\mu \cB-\cM
\end{equation}
where $\cP$ depends nonlinearly on $\cE$. In the absence of charges, currents and magnetization, i.~e.\ $\cJ=\cM=0$, $\rho=0$, using \eqref{eq:relations}, and differentiating the first equation in \eqref{eq:Maxwell} with respect to $t$, we arrive at the equation
$$
\curlop\left(\frac1\mu\curlop \cE\right)+\eps\partial_t^2 \cE
 = -\partial_t^2 \cP.
$$
In the time-harmonic case the fields $\cE$ and $\cP$ are of the form
$\cE(x,t) = E(x)e^{i\omega t}$, $\cP(x,t) = P(x)e^{i\omega t}$,
which leads to the time-harmonic Maxwell equation
$$
\curlop\left(\frac1\mu\curlop E\right)-\omega^2\eps E = \omega^2 P.
$$
Since $P$ depends on $E$, and assuming $\eps,\mu>0$ to be constant, we finally obtain an equation of the form
\begin{equation}\label{eq:main}
\curlop(\curlop E) + \la E = f(x,E) \qquad\textnormal{in } \Om,
\end{equation}
where $\la = -\mu\omega^2\eps\leq 0$.
In a Kerr-like medium one has $\cP = \al(x)|\cE|^2\cE$, hence
\[
f(x,E)=\mu\omega^2\al(x) |E(x)|^2 E(x).
\]
We shall treat more general nonlinearities $f(x,E)=\pa_E F(x,E)$, having $F(x,E)=\frac1p\Ga(x)|E|^p$ with $\Ga\in L^\infty(\Om)$ positive and $2<p<6$ as model in mind. The goal of this paper is to solve \eqref{eq:main} together with the boundary condition
\begin{equation}\label{eq:bc}
\nu\times E = 0\qquad\text{on }\pa\Om
\end{equation}
where $\nu:\pa\Om\to\R^3$ is the exterior normal. This boundary condition holds when $\Om$ is surrounded by a perfect conductor.

There is very little work on semilinear equations involving the curl-curl operator $(\curlop)^2:E\mapsto\curlop(\curlop E)$. One difficulty from a mathematical point of view is that the curl-curl operator has an infinite-dimensional kernel, namely all gradient vector fields. Solutions of \eqref{eq:main} are critical points of the functional
\begin{equation}\label{eq:action}
\fJ(E) = \frac12\int_\Om|\curlop E|^2dx + \frac\la2\int_\Om |E|^2dx - \int_\Om F(x,E)\,dx
\end{equation}
defined on an appropriate subspace of $H_0(\curl;\Om)$; see Section~3 for the definition of the spaces we work with. Using the Helmholtz decomposition $E=v+\nabla w$ with $\div v=0$ we have
\begin{equation*}
J(v,w) := \fJ(v+\nabla w)
 = \frac12\int_\Om|\curlop v|^2dx + \frac\la2\int_\Om |v+\nabla w|^2dx
    - \int_\Om F(x,v+\nabla w)\,dx.
\end{equation*}
For the nonlinearities which we consider, the functional is unbounded from above and from below. Moreover, critical points have infinite Morse index.

Take for instance the model case $F(x,E)=\frac14|E|^4$. Then we have on the space of divergence-free vector fields $v\in H_0(\curl;\Om)$:
\begin{equation*}
\begin{aligned}
J(v,0)=\fJ(v)
 &= \frac12\int_\Om|\curlop v|^2dx +  \frac\la2\int_\Om |v|^2dx
    - \frac14\int_\Om |v|^4dx\\
 &= \frac12\int_\Om|\nabla v|^2dx +  \frac\la2\int_\Om |v|^2dx - \frac14\int_\Om |v|^4dx,
\end{aligned}
\end{equation*}
and on the space of gradient vector fields $\nabla w\in L^4(\Om,\R^3)$:
\begin{equation*}
J(0,w)=\fJ(\nabla w)
 = \frac\la2\int_\Om |\nabla w|^2dx - \frac14\int_\Om |\nabla w|^4dx.
\end{equation*}
Thus if $-\De+\la>0$ on $H^1_0(\Om)$, then $E=0$ is a local minimum of $\fJ$ on the space of divergence-free vector fields $v\in H_0(\curl;\Om)$, but is not a local minimum on the space of gradient vector fields $\nabla w\in L^4(\Om)$, not even if $\la>0$. In fact, since the $L^2$-norm is weaker than the $L^4$-norm $0$ is a degenerate critical point of $\fJ$ on the space of gradient vector fields, neither a local minimum nor a local maximum (except $\la\le0$), nor a saddle point. The mountain pass value is $0$, and the Palais-Smale sequences associated to the mountain pass value will converge to $0$ even when $\la$ is positive. If $\la\le0$ then $0$ is a local maximum for $\fJ$ on the space of gradient vector fields, and if even $-\De+\la\le0$ it is a, possibly degenerate, saddle point for $\fJ$ on the space of divergence-free vector fields.

In addition to these problems related to the geometry of $\fJ$ and $J$, we also have to deal with compactness issues although $\Om$ is assumed to be bounded. Namely in all treatments of strongly indefinite functionals $J:X\to \R$ it is required that $J$ is essentially of the form
\[
J(u) = \frac12\left(\|u^+\|_X^2-\|u^-\|_X^2\right) - I(u)
\]
with $I'$ (sequentially) weak-to-weak$^*$ continuous; see e.~g.\ \cite{BenciRabinowitz,BartschDing,DingBook}.
However, we have to deal with a functional of the form
\[
J(u) = \frac12\|u^+\|_X^2 - I(u)
\]
where $I':X\to X^*$ is not (sequentially) weak-to-weak$^*$ continuous. Therefore we do not know whether a weak limit of a bounded Palais-Smale sequence is a critical point.
Although the methods from \cite{BartschDing} allow to deal with an infinite-dimensional kernel, the weak-to-weak$^*$ continuity of $I'$ is essential.

We shall use two approaches to find critical points of $\fJ$ and $J$ which work under different hypotheses on $F$ and $\Om$. One approach uses a generalization to strongly indefinite functionals of the Nehari manifold technique due to Pankov \cite{Pankov}; see \cite[Chapter~4]{SzulkinWethHandbook} for a survey. For this approach we cannot just cite existing results due to the above mentioned problems.
The other one is based on the the Palais principle of symmetric criticality. Roughly speaking, if $\Omega$ and $F$ are cylindrically symmetric, then $\fJ$ may be restricted to a subspace, where we are able to apply standard critical point theory.

In \cite{BenFor} Benci and Fortunato introduce a model for a unified field theory for classical electrodynamics which is based on a semilinear perturbation of the Maxwell equations. In the magnetostatic case, in which the electric field vanishes and the magnetic field is independent of time, they are lead to an equation of the form
\begin{equation}\label{eq:benci-fortunato}
\curlop(\curlop A) = W'(|A|^2)A\qquad\text{in } \R^3
\end{equation}
for the gauge potential $A$ related to the magnetic field $H=\curlop A$. Here
$F(A)=\frac12 W(|A|^2)$ is superquadratic and subcritical, and satisfies various additional conditions which exclude the model $F(A)=\frac14|A|^4$. In \cite{BenForAzzAprile} Azzollini et al.\ use the symmetry of the domain and of the equation to find solutions of \eqref{eq:benci-fortunato} of the form
\begin{equation}\label{eq:sym1}
A(x)=\al(r,x_3)\begin{pmatrix}-x_2\\x_1\\0\end{pmatrix},\qquad r=\sqrt{x_1^2+x_2^2}.
\end{equation}
A field of this form is divergence-free, so that the functional has the form
\[
\fJ(A) = \frac12\int_{\R^3}|\nabla A|^2 - W(|A|^2)\,dx
\]
hence standard methods of nonlinear analysis apply. In \cite{DAprileSiciliano} D'Aprile and Siciliano find solutions of \eqref{eq:benci-fortunato} of the form
\begin{equation}\label{eq:sym2}
A(x)=\be(r,x_3)\begin{pmatrix}x_1\\x_2\\0\end{pmatrix}
      +\ga(r,x_3)\begin{pmatrix}0\\0\\1\end{pmatrix},\qquad r=\sqrt{x_1^2+x_2^2},
\end{equation}
again using symmetry arguments.
When the domain has a cylindrical symmetry as in \cite{BenForAzzAprile} we can modify the approach from \cite{BenForAzzAprile} and obtain symmetric solutions. We would like to emphasize that we can also deal with nonsymmetric domains and with functions $F(x,E)$ that depend on $x$ and are not radial in $E$.

Finally we would like to mention the papers \cite{Stuart91,Stuart04,StuartZhou96, StuartZhou01,StuartZhou03,StuartZhou05, StuartZhou10} by Stuart and Zhou, who studied transverse electric and transverse magnetic solutions to (\ref{eq:Maxwell}) in $\R^3$.

The paper is organized as follows. In the next section we formulate our hypotheses on $\Om$ and $F$, and we state our main results concerning the existence of a ground state and of bound states. In Section~\ref{sec:setting} we introduce the variational setting, in particular the spaces on which $\fJ$ and $J$ will be defined. Next, in Section~\ref{sec:Nehari} we present some critical point theory for a class of functionals like $J$, and we introduce the Nehari-Pankov manifold on which we minimize $J$ to find a ground state. Finally, in Sections~\ref{sec:proof-main}-\ref{sec:proof-sym} we prove our results.

\section{Statement of results}\label{sec:results}

Let $\Om\subset\R^3$ be a simply connected, bounded domain with connected $C^{1,1}$-boundary, or let $\Om$ be a bounded and convex domain of $\R^3$. We want to find weak solutions $E:\Om\to\R^3$ of the boundary value problem
\begin{equation}\label{eq:problem}
\curlop(\curlop E) + \la E = f(x,E)\quad \textnormal{ in } \Om,\qquad
\nu\times E=0 \quad \textnormal{ on } \pa\Om.
\end{equation}
The boundary condition has to be understood in a weak sense for $\Om$ being convex with non-smooth boundary.

We collect assumptions on the nonlinearity $F(x,u)$.
\begin{itemize}
\item[(F1)] $F:\Om\times\R^3\to\R$ is differentiable with respect to $u\in\R^3$, and $f=\pa_uF:\Om\times\R^3\to\R^3$ is a Carath\'eodory function (i.~e.\ measurable in $x\in\Om$, continuous in $u\in\R^3$ for a.~e.\ $x\in\Om$).
\item[(F2)] $|f(x,u)|=o(|u|)$ as $u\to0$ uniformly in $x\in\Om$.
\item[(F3)] There exist constants $c>0$, $2<p<6$, such that
    \[|f(x,u)|\le c(1+|u|^{p-1})\qquad\text{for all } x\in\Om, u\in\R^3.\]
\item[(F4)] There exists a constant $d>0$, such that for all $x\in\Om$ and all $u\in\R^3\setminus\{0\}$
    \[\frac12\langle f(x,u),u\rangle > F(x,u)\geq d|u|^p.\]
\item[(F5)] $F$ is convex with respect to $u\in\R^3$.
\item[(F6)] $F$ is strictly convex with respect to $u\in\R^3$ if $-\la$ is an eigenvalue of the curl-curl operator in the space of divergence-free vector fields satisfying the boundary condition \eqref{eq:bc}, i.~e.\ if the eigenvalue problem
\begin{equation}\label{EgEigenvalue}
\left\{
\begin{array}{ll}
    \curlop(\curlop u) + \la u = 0,\; \div(u)=0
    &
    \hbox{in } \Om,\\
    \nu\times u =0
    &
    \hbox{on } \partial \Om.
\end{array}
\right.
\end{equation}
has a solution $u\ne0$. If $\la = 0$ then $F$ is uniformly strictly convex with respect to $u\in\R^3$, i.e.\ for any compact $A\subset(\R^3\times\R^3)\setminus\{(u,u):\;u\in\R^3\}$
$$
\inf_{\genfrac{}{}{0pt}{}{x\in\R^3}{(u_1,u_2)\in A}}
 \left(\frac12\big(F(x,u_1)+F(x,u_2)\big)-F\left(x,\frac{u_1+u_2}{2}\right)\right) > 0.
$$
\item[(F7)] If $ \langle f(x,u),v\rangle = \langle f(x,v),u\rangle \ne 0\ $ then
$\ \displaystyle F(x,u) - F(x,v)
 \le \frac{\langle f(x,u),u\rangle^2-\langle f(x,u),v\rangle^2}{2\langle f(x,u),u\rangle}$.\\
If in addition $F(x,u)\ne F(x,v)$ then the strict inequality holds.
\end{itemize}

Conditions (F1)-(F3) are rather harmless, and condition (F4) is reminiscent of the Ambro\-setti-Rabinowitz condition. The convexity conditions (F5) and (F6) are needed both for semi-continuity and linking. The technical condition (F7) will be needed to set up the  Nehari-Pankov manifold.

\begin{Rem}
a) If $F$ is of the form $F(x,u)=\Ga(x)|Mu|^p$ with $\Ga\in L^\infty(\Om)$ positive and bounded away from $0$, $M\in GL(3)$ an invertible $3\times 3$ matrix, and $2<p<6$, then all assumptions on $F$ are satisfied. Also sums of such functions are allowed. Observe that these functions are not radial when $M$ is not an orthogonal matrix.

b) Suppose $F$ is a radial map with respect to $u$, i.e. $F(x,u)=W(x,|u|^2)$, and
$W=W(x,t):\Om\times [0,\infty)\to\R$ is differentiable with respect to $t$ with $\pa_tW:\Om\times\R\to\R$ a Carath\'eodory function, and suppose in addition that $\pa_tW(x,t)$ is strictly increasing in $t\in(0,\infty)$, $W(x,0)=0$, for any $x\in\R^3$. Then it is not difficult to check that (F1), (F5) and (F7) hold.
Moreover $\frac{1}{2}\langle f(x,u), u\rangle > F(x,u)$ for $x\in\Om$ and $u\in\R^3\setminus\{0\}$. Indeed, note that for $u\in\R^3$
\begin{eqnarray*}
F(x,u)
 &=& F(x,u)-F(x,0)
  = \int_0^{1}\frac{d}{ds}F(x,su)\, ds
  = \int_0^{1}\pa_t W(x,|su|^2)2s|u|^2\ ds\\
 &\le& \pa_tW(x,|u|^2)|u|^2\int_0^{1}2s\ ds=\frac12\langle\pa_u F(x,u),u\rangle.
\end{eqnarray*}
and the last inequality is strict if $u\neq 0$.
\end{Rem}

\begin{Th}\label{thm:main}
Suppose the assumptions (F1)-(F7) hold. If $\la\le0$ then \eqref{eq:problem} has a ground state solution $E=v+\nabla w$ with $\div v=0$ and $v\ne0$. If in addition $F$ is even in $u$ then \eqref{eq:problem} has infinitely many solutions.
\end{Th}

A ground state solution is a minimizer of the associated energy functional on the Nehari-Pankov manifold which will be defined in Section~4.

In a symmetric setting as in \cite{BenForAzzAprile} more can be said about the shape of the solutions. More precisely, we set $G=\O(2)\times\{1\}\subset O(3)$ and require:
\begin{itemize}
\item[(S)] $\Om$ is invariant with respect to $G$. $F$ is invariant with respect to the action of $G$ on the $x$-variable, and $F$ is radially symmetric with respect to $u$: $F(x,u)=F(gx,|u|)$ for all $x\in\Om$, $u\in\R^3$, $g\in G$.
\end{itemize}

Condition (S) simplifies the problem considerably. It allows to find solutions having a cylindrical symmetry as in \cite{BenForAzzAprile}. Here we assume the Ambrosetti-Rabinowitz-type condition
\begin{itemize}
\item[(F8)] There exists a constant $\theta>2$ such that
    \[
    \frac1\theta\langle f(x,u),u\rangle \ge F(x,u) > 0\qquad
    \text{for all $x\in\Om$ and all $u\in\R^3\setminus\{0\}$.}
    \]
\end{itemize}

\begin{Th}\label{thm:sym1}
Suppose (S) and (F1)-(F3), (F8) hold. Then \eqref{eq:problem} has infinitely many solutions of the form
\begin{equation}\label{eq:sym-sol}
E(x)=\al(r,x_3)\begin{pmatrix}-x_2\\x_1\\0\end{pmatrix},\qquad r=\sqrt{x_1^2+x_2^2}.
\end{equation}
\end{Th}

Observe that we do not require $\la\le0$, in fact here we could even replace $\la$ by a potential $V\in L^\infty(\Om)$.

\begin{Rem}
In the proof of Theorem \ref{thm:sym1} we use (S) to reduce the problem to that of finding critical points of $\fJ$ constrained to a Hilbert space $Y$ of fields $E$ of the form \eqref{eq:sym-sol}. For $E\in Y$ the functional has the form
\[
\fJ(E) = \frac12\|E\|^2 + \frac\la2\int_{\Om} |E|^2\,dx - \int_{\Om}F(x,E)\,dx.
\]
It is then standard to obtain critical points of $\fJ|_Y$ under various hypotheses on $\la$ and $F$. We chose the Ambrosetti-Rabinowitz condition but the interested reader may play with other types of nonlinearities.
\end{Rem}

\section{The variational setting}\label{sec:setting}

The natural space for the Maxwell eigenvalue problem \eqref{EgEigenvalue} is the space
$$
H(\curl;\Om) := \{E\in L^2(\Om,\R^3): \curlop E \in L^2(\Om,\R^3)\}
$$
where the curl of $E$, $\curlop E$, has to be understood in the distributional sense. This is a Hilbert space when provided with the graph norm
$$
\|E\|_{H(\curl;\Om)} := \left(|E|^2_2+|\curlop E|^2_2\right)^{1/2}.
$$
Here and in the sequel $|\cdot|_q$ denotes the $L^q$-norm. The closure of $C^{\infty}_0(\Om,\R^3)$ in $H(\curl;\Om)$ is denoted by $H_0(\curl;\Om)$. There is a continuous tangential trace operator $\gamma_t:H(\curl;\Om) \to H^{-1/2}(\pa\Om)$ such that
$$
\gamma_t(E)=\nu\times E_{|\pa\Om}\qquad\text{for any $E\in C^{\infty}(\overline\Om,\R^3)$}
$$
and (see \cite[Theorem~3.33]{Monk})
$$
H_0(\curl;\Om)=\{E\in H(\curl;\Om): \gamma_t(E)=0\}.
$$

The subspace of divergence-free vector fields is defined by
\[
\begin{aligned}
\cV
 &= \left\{E\in H_0(\curl;\Om): \int_\Om\langle E,\nabla w\rangle\;dx=0
        \text{ for any }w\in H^1_0(\Om)\right\}\\
 &= \{E\in H_0(\curl;\Om): \div E=0\}
\end{aligned}
\]
where $\div E$ has to be understood in the distributional sense. Then $\cV$ and the subspace of gradient vector fields $\nabla H^1_0(\Om):=\{\nabla w: w\in H^1_0(\Om)\}$ are orthogonal in $L^2(\Om,\R^3)$ and in $H_0(\curl,\Om)$; moreover
$$
H_0(\curl,\Om)=\cV\oplus \nabla H^1_0(\Om).
$$
Since $\Om$ has a $C^{1,1}$ boundary, or is a convex domain, there exists a continuous embedding
$$
\cV \subset \{E\in H_0(\curl;\Om): \div E\in L^2(\Om)\} \hookrightarrow H^1(\Om,\R^3);
$$
see \cite[Theorem~2.12 and Theorem~2.17]{Amrouche} or \cite[Theorem~1.2.1]{Doerfler}.
Therefore in view of Rellich's theorem $\cV$ is compactly embedded in $L^q(\Om,\R^3)$ for $1\le q<6$. Due to the embedding $\cV\hookrightarrow L^2(\Om;\R^3)$ the norm
$$
\|v\|_{\cV} := |\curlop v|_2,\quad v\in\cV,
$$
is equivalent to $\|\cdot\|_{H(\curl;\Om)}$ on $\cV$. Moreover, $\cV$ is a Hilbert space with the scalar product
$$
\langle u,v \rangle := \int_{\Om}\langle \curlop u,\curlop v\rangle\; dx.
$$

The spectrum of the curl-curl operator in $H_0(\curl;\Om)$ consists of the eigenvalue $0$ with infinite multiplicity and eigenspace $\nabla H^1_0(\Om)$, and of a sequence of eigenvalues $0<\la_1\le \la_2\le\dots\le\la_k\to\infty$ with finite multiplicities and eigenfunctions in $\cV$; see \cite[Theorem~4.18]{Monk}. In fact, for any $f\in L^2(\Om,\R^3)$ the equation
\begin{equation}\label{eq:operator}
\curlop(\curlop v) + v = f
\end{equation}
has a unique solution $v\in\cV$ and the operator
\[
K:L^2(\Om,\R^3) \to \cV\subset L^2(\Om,\R^3),\quad
 Kf=v\text{ solves \eqref{eq:operator},}
\]
is compact and self-adjoint.

\begin{Rem}\label{rem:spectrum}
Given $V\in L^\infty(\Om)$ the operator $(\curlop)^2 + V(x)$ in $H_0(\curl;\Om)$ is not a compact perturbation of $(\curlop)^2$. The treatment of nonconstant potentials requires a somewhat different variational setting. This is work in progress.
\end{Rem}

We also need the space
$$
W^p(\curl;\Om) := \{E\in L^p(\Om,\R^3)|\; \curlop E\in L^2(\Om,\R^3)\}
$$
with the norm
$$
\|E\|_{W^p(\curl;\Om)} := \left(|E|^2_p+|\curlop E|^2_2\right)^{1/2}
$$
and let $W^p_0(\curl;\Om)$ be the closure of $C^{\infty}_0(\Om)$ in $W^p(\curl;\Om)$. Note that $W^p(\curl;\Om)$ and $W_0^p(\curl;\Om)$ are Banach spaces, $\cV$ is a closed subspace of $W_0^p(\curl;\Om)$, and
$$
W_0^p(\curl;\Om) \subset H_0(\curl;\Om) = W_0^2(\curl;\Om).
$$
Setting
$\cW := W^{1,p}_0(\Om)$ we see that
\[
\nabla\cW := \{\nabla w: w\in\cW\} \subset W_0^p(\curl;\Om)
\]
is a closed subspace of $W_0^p(\curl;\Om)$ and
$$
W_0^p(\curl;\Om) = \cV\oplus\nabla\cW.
$$
In $\cW$ we introduce the norm
$$
\|w\|_{\cW} := |\nabla w|_p
$$
and observe that $\|w\|_{\cW} = \|\nabla w\|_{W^p(\curl,\Om)}$ because $\curlop\nabla w = 0$. Clearly $\nabla \cW$ is a closed subspace of $L^p(\Om,\R^3)$, and $\nabla \cW$, $\cW$, $W_0^p(\curl;\Om)$ are separable and reflexive Banach spaces. We use the following norm in $\cV\times\cW$:
$$
\|(v,w)\| := \left(\|v\|_{\cV}^2+\|w\|^2_{\cW}\right)^{1/2}
 = \left(|\curlop v|_2^2+|\nabla w|^2_p\right)^{1/2}
\qquad\text{for }(v,w)\in\cV\times\cW
$$

By our assumptions on $F$ the functional $\fJ: W_0^p(\curl;\Om) \to \R$,
\[
\fJ(E) := \frac12\int_\Om|\curlop E|^2dx + \frac\la2\int_\Om |E|^2dx - \int_\Om F(x,E)\,dx
\]
and the functional $J:\cV\times\cW\to\R$, defined by
\[
\begin{aligned}
J(v,w)
 &:= \fJ(v+\nabla w)
  = \frac12\int_{\Om}|\curlop v|^2\;dx + \frac\la2\int_\Om |v+\nabla w|^2\;dx
    - \int_{\Om}F(x,v+\nabla w)\;dx\\
 &= \frac12\|v\|_\cV^2 + \frac\la2\int_\Om (|v|^2+|\nabla w|^2)\;dx
    - \int_{\Om}F(x,v+\nabla w)\;dx,
\end{aligned}
\]
are well defined and of class $\cC^1$ with
\[
\begin{aligned}
&\fJ'(v+\nabla w)(\phi+\nabla\psi)
  = J'(v,w)(\phi,\psi)\\
 &\hspace{1cm}
  = \int_\Om\langle \curlop v,\curlop \phi\rangle\; dx
   + \la\int_\Om(\langle v,\phi\rangle+\langle\nabla w,\nabla\psi\rangle)\; dx
   - \int_\Om\langle f(x,v+\nabla w),\phi+\nabla\psi\rangle \; dx
\end{aligned}
\]
for any $(v,w),(\phi,\psi)\in \cV\times\cW$.

Our preceding discussion yields

\begin{Prop}\label{PropSolutE}
$(v,w)\in \cV\times\cW$ is a critical point of $J$ if and only if
$E=v+\nabla w\in W_0^p(\curl;\Om)=\cV\oplus\nabla\cW$ is a critical point of $\fJ$, hence a solution of \eqref{eq:problem}.
\end{Prop}

Observe that in the case of our model nonlinearity $F(x,u)=\frac1p|u|^p$ the functionals are of class $\cC^2$ with second variation of $J$ at $(0,0)$ given by
\begin{equation*}
J''(0,0)[(\phi,\psi),(\phi,\psi)]
 = \|\phi\|^2_\cV + \la\int_\Om(|\phi|^2+|\nabla\psi|^2)dx.
\end{equation*}
Since the $L^2$-norm is weaker than the $L^p$-norm, $J''(0)$ is not positive definite on $\cV\times\cW$. Even if $\la>0$ is positive, $J$ does not satisfy the mountain pass geometry. Setting
\[
\begin{aligned}
I(v,w)
 &:= -\frac\la2\int_{\Om}|v+\nabla w|^2\;dx + \int_{\Om}F(x,v+\nabla w)\;dx\\
 &= -\frac\la2\int_{\Om}(|v|^2+|\nabla w|^2)\;dx + \frac1p|v+\nabla w|_p^p
\end{aligned}
\]
we see that $I'$ is not (sequentially) weak-to-weak$^*$ continuous because
$v_n+\nabla w_n\weakto v+\nabla w$ in $L^p(\Om;\R^3)$ does not imply
$|v_n+\nabla w_n|^{p-2}(v_n+\nabla w_n)\weakto |v+\nabla w|^{p-2}(v+\nabla w)$ in $L^{p'}$. 

\section{Critical point theory and the Nehari-Pankov manifold}\label{sec:Nehari}

Let $X$ be a reflexive Banach space with norm $\|\cdot\|$ and with a topological direct sum decomposition $X = X^+\oplus \tX$. For $u \in X$ we denote by $u^+ \in X^+$ and $\tu  \in \tX$ the corresponding summands so that $u = u^++\tu $. We may also assume that $\|u\|^2 = \|u^+\|^2+\|\tu \|^2$. Moreover we assume that
$X^+\ni u \mapsto \|u\|^2\in\R$ is a $\cC^1$-map, so that the unit sphere
$S^+ := \{u\in X^+:\|u\|=1\}$ in $X^+$ is a $\cC^1$-submanifold of $X^+$. In addition to the norm topology we need the topology $\cT$ on $X$ which is the product of the norm topology in $X^+$ and the weak topology in $\tX$. In particular, $u_n\cTto u$ provided that $u_n^+ \to u^+$ and $u_n' \weakto \tu $. On bounded subsets of $X$ the topology $\cT$ coincides with the metrizable topology considered by Bartsch and Ding \cite{BartschDing} and for Hilbert spaces by Kryszewski and Szulkin \cite{KryszSzulkin}.

We consider a functional $J\in\cC^1(X,\R)$ of the form
\begin{equation}\label{EqJ}
J(u) = \frac12\|u^+\|^2-I(u) \quad\text{for $u=u^++\tu \in X^+\oplus \tX$}
\end{equation}
such that the following assumptions hold:
\begin{itemize}
\item[(A1)] $I\in\cC^1(X,\R)$ and $I(u)\ge I(0)=0$ for any $u\in X$.
\item[(A2)] $I$ is $\cT$-sequentially lower semicontinuous:
    $u_n\cTto u\quad\Longrightarrow\quad \liminf I(u_n)\ge I(u)$
\item[(A3)]  If $u_n\cTto u$ and $I(u_n)\to I(u)$ then $u_n\to u$.
\end{itemize}
The geometry of $J$ is described by the following assumptions; see \cite[Section~4]{SzulkinWethHandbook}. For $u\in X\setminus \tX$ let
\[
X(u):=\R u\oplus \tX\qquad\text{and}\qquad\wh{X}(u) := \R^+u\oplus \tX
\]
where $\R^+=[0,\infty)$.
\begin{itemize}
\item[(A4)] There exists $r>0$ such that $a:=\inf\limits_{u\in X^+:\|u\|=r} J(u)>0$.
\item[(A5)] For every $u\in X\setminus \tX$ there exists a unique critical point
    $0\ne \wh{m}(u)\in\wh{X}(u)$ of $J|_{X(u)}$. Moreover, $\wh{m}(u)$ is the unique global maximum of $J|_{\wh{X}(u)}$.
\item[(A6)] There exists $\de>0$ such that $\|\wh{m}(u)^+\|\ge\de$ for all $u\in X\setminus \tX$. Moreover, $\wh{m}$ is bounded on compact subsets of $X\setminus \tX$.
\end{itemize}

Now we define the Nehari-Pankov manifold
\[
\cN := \{\wh{m}(u): u\in X\setminus \tX\} \stackrel{(A5)}{=} \{u\in X\setminus \tX: J'(u)|_{X(u)}=0\}.
\]
This manifold has been introduced by Pankov \cite{Pankov}. In the case $X=X^+$, $\tX=0$, it yields the well known Nehari manifold.
$\cN$ is homeomorphic to $S^+$, but observe that $\cN$ need not be differentiable because $J$ is only of class $\cC^1$. One can show that a critical point of $J|_\cN$, in the sense of metric critical point theory as in \cite{Corvellec-Degiovanni-Marzocchi:1993}, is a critical point of $J$. At first sight it is not even clear that a minimizer of $J$ on $\cN$ is a critical point of $J$. Observe that $\inf_\cN J > J(0)=0$ by (A4)-(A5).

We say that $J$ satisfies the $(PS)_c^\cT$-condition in $\cN$ if every $(PS)_c$-sequence in $\cN$ has a subsequence which converges in $\cT$:
\[
u_n \in \cN,\ J'(u_n) \to 0,\ J(u_n) \to c \qquad\Longrightarrow\qquad
u_n \cTto u\in X\ \text{ along a subsequence}
\]

\begin{Th}\label{ThLink1}
Let $J \in \cC^1(X,\R)$ satisfy (A1)-(A6) and set $c_0 = \inf_\cN J$. Then the following holds:

a) $J$ has a $(PS)_{c_0}$-sequence in $\cN$.

b) If $J$ satisfies the $(PS)_{c_0}^\cT$-condition in $\cN$ then $c_0$ is achieved by a critical point of $J$.

c) If $J$ satisfies the $(PS)_c^\cT$-condition in $\cN$ for every $c$ and if $J$ is even then it has an unbounded sequence of critical values.
\end{Th}

\begin{proof}
As in \cite[Section~4]{SzulkinWethHandbook} one proves that
\begin{itemize}
\item[(i)] $m := \wh{m}|_{S^+}:S^+\to\cN$ is a homeomorphism with inverse $\cN \to S^+$, $u \mapsto u^+/\|u^+\|$.
\item[(ii)] $J\circ m:S^+\to\R$ is $\cC^1$.
\item[(iii)] $(J\circ m)'(u) = \|m(u)^+\|\cdot J'(u)|_{T_uS^+}:T_uS^+\to\R$ for every $u\in S^+$.
\item[(iv)] $(u_n)_n\subset S^+$ is a Palais-Smale sequence for $J\circ m$ if, and only if, $(m(u_n))_n$ is a Palais-Smale sequence for $J$ in $\cN$.
\item[(v)] $u\in S^+$ is a critical point of $J\circ m$ if, and only if, $m(u)$ is a critical point of $J$.
\item[(vi)] If $J$ is even, then so is $J\circ m$.
\end{itemize}
The existence of a $(PS)_{c_0}$-sequence $(u_n)_n$ for $J$ in $\cN$ follows from (ii) and (iv) because $c_0 = \inf J\circ m$. We claim that
\begin{itemize}
\item[(vii)] If $J$ satisfies the $(PS)_c^\cT$-condition in $\cN$ for some $c>0$ then $J\circ m$ satisfies the $(PS)_c$-condition.
\end{itemize}
In order to see this consider a $(PS)_c$-sequence $(u_n)_n$ for $J\circ m$. Then $(m(u_n))_n$ is a Palais-Smale sequence for $J$ in $\cN$ by (iv), hence $m(u_n)\cTto v$ after passing to a subsequence. This implies $m(u_n)^+\to v^+$ and moreover, using (A2),
\[
0<c = \lim_{n\to\infty}J(m(u_n)) \leq J(v).
\]
Now (A1) implies $v^+\neq 0$, hence $m(u_n)^+\ne0$ for $n$ large. From the continuity of $m$ we deduce
\[
m(u_n) = m(m(u_n)^+/\|m(u_n)^+\|)\to m(v^+/\|v^+\|),
\]
and therefore $v=m(v^+/\|v^+\|)\in\cN$ and $m(u_n)\to v$. It follows that
\[
u_n = m(u_n)^+/\|m(u_n)^+\| \to v^+/\|v^+\|.
\]
This proves (vii).\\
Next observe that if $J$ satisfies the $(PS)_{c_0}^\cT$-condition in $\cN$ then $c_0$ is achieved by a critical point $u\in S^+$ of $J\circ m$, hence $m(u)\in\cN$ is a critical point of $J$ with $J(m(u))=c_0$. This proves b).\\
Finally c) follows from standard Ljusternik-Schnirelman theory. Under the conditions of c) the functional $J\circ m$ is even, bounded below, and satisfies the Palais-Smale condition. Hence it has an unbounded sequence of critical values, and so does $J$ by (v).
\end{proof}

Assumptions (A5)-(A6) can be checked with the help of the following conditions.
\begin{itemize}
\item[(B1)] $\|u^+\|+I(u)\to\infty$ as $\|u\|\to\infty$.
\item[(B2)] $I(t_nu_n)/t_n^2\to\infty$ if $t_n\to\infty$ and $u_n^+\to u^+$ for some $u^+\neq 0$ as $n\to\infty$.
\item[(B3)] $\frac{t^2-1}{2}I'(u)[u] + tI'(u)[v] + I(u) - I(tu+v)<0$ for every $u\in X$, $t\ge 0$, $v\in \tX$ such that $u\neq tu+v$.
\end{itemize}

\begin{Prop}\label{prop:geometry}
Let $J\in\cC^1(X,\R)$ satisfy (A1)-(A2), (A4), (B1)-(B3) and let $X^+$ be a Hilbert space with the scalar product such that $\langle u,u\rangle=\|u\|^2$ for any $u\in X^+$. Then (A5) and (A6) hold.
\end{Prop}
\begin{proof}
Let $u\in X\setminus \tX$. Assume that
$t_nu+\tu _n\weakto t_0u+\tu _0$ as $n\to\infty$ where $\tu_n\in \tX$, $t_n\geq 0$ for $n\geq 0$. Then $t_n\to t_0$ and from the $\cT$-sequentially lower semicontinuity of $I$ we obtain
\begin{eqnarray*}
\liminf_{n\to\infty}(-J(t_nu+\tu _n))\geq -J(t_0u+\tu _0).
\end{eqnarray*}
Therefore $-J$ is weakly sequentially lower semi-continuous on $\wh{X}(u)$.
Moreover, $-J$ is coercive on $\wh{X}(u)$ by (B1)-(B2). Thus there exists a global maximum $\wh{m}(u)\in \wh{X}(u)$ of $J|_{\wh{X}(u)}$. In view of (A4) we easily see that $J(\wh{m}(u))\geq a >0$ and thus $\wh{m}(u)\notin\tX$.
Therefore $\wh{m}(u)$ is a critical point of $J|_{X(u)}$. Now we show the uniqueness of $\wh{m}(u)$. Let
$u\in X\setminus \tX$ be any critical point of $J|_{X(u)}$. Observe that for any $t\geq 0$ and $v\in \tX$
$$
J(t u+v)-J(u) = \frac{t^2-1}{2}\|u^+\|^2 + I(u) - I(tu+v).
$$
Then by (B3)
$$
J(t u+v)-J(u) = \frac{t^2-1}{2}I'(u)[u] + tI'(u)[v] + I(u) - I(tu+v) < 0
$$
provided that $u\neq tu+v$. Therefore $u$ is the unique global maximum of $J|_{X(u)}$ and (A5) holds.

In order to prove (A6) observe that the first statement follows from (A4). For the second statement we claim that for $u_0\in X\setminus\tX$ there exists $R>0$ such that $J\le0$ on $\wh{X}(u)\setminus B(0,R)$ for $u\in X\setminus\tX$ close to $u_0$. If not there exists a sequence $t_nu_n^+ + v_n \in \R^+u_n^+\oplus\tX$ such that $X^+ \ni u_n^+ \to u_0^+ \ne 0$, $\|t_nu_n+v_n\| \to \infty$ and $J(t_n u_n^+ + v_n) > 0$. Then by (B1) and using
$$
\frac12 t_n^2\|u_n^+\|^2 > I(t_n u_n^+ + v_n),
$$
we deduce $t_n\to\infty$. Therefore $(B2)$ implies
$$
J(t_nu_n^+ + v_n)
 = t_n^2\Big(\frac12\|u_n^+\|^2-\frac{I(t_n(u_n^+ + v_n/t_n))}{t_n^2}\Big)
 \to -\infty,
$$
which contradicts $J(t_n u_n^+ + v_n) > 0$. This proves the claim, from which we can then deduce that $\|\wh{m}(u)\| \le R$ for $u$ close to $u_0$ because $J(\wh{m}(u))\geq a>0$.
\end{proof}

Observe that (A1), (A4) and (B1)-(B2) imply the linking geometry of $J$, i.e. for any $u^+\in X^+$ there are $R>r>0$ such that
\begin{equation*}
\sup_{\partial M(u^+)} J\leq 0=J(0)<\inf_{S^+_r}J
\end{equation*}
where
\begin{eqnarray*}
\partial M(u^+)&:=&\{u=t u^++\tu \in X |\; \tu \in \tX,\;
(\|u\|=R,t\geq 0)\textnormal{ or }(\|u\|\leq R,t=0) \},\\
S^+_r&:=&\{ u^+\in X^+,\|u^+\|=r\}.
\end{eqnarray*}

\section{Proof of the Theorem~\ref{thm:main}}\label{sec:proof-main}

Recall that the spectrum of \eqref{EgEigenvalue} is discrete and consists of an unbounded sequence of eigenvalues $0<\la_1\le\la_2\le\dots\le\la_k\to\infty$ with finite multiplicities. Let $e_k\in\cV$ be the eigenfunction corresponding to $\la_k$, and recall that these eigenfunctions are orthogonal with respect to the scalar products in $H(\curl;\Om)$ and  $L^2(\Om;\R^3)$. Let
\[
n:=\min\{k\in\N_0:\la_{k+1}>0\} = \max\{k\in\N_0:\la_k\le0\}
\]
be the dimension of the semi-negative eigenspace; here $\la_0:=-\infty$. The quadratic form $Q(v):=\int_\Om \left(|\curlop v|^2+\la|v|^2\right)$ is positive definite on the space
\[
\cV^+ := \span\{e_k:k>n\},
\]
and it is negative semidefinite on
\[
\tcV := \span\{e_1,\dots,e_n\}.
\]
Here $\tcV=0$ if $n=0$, of course. Observe that there exists $\de>0$ such that
\begin{equation}\label{NormInU0}
\int_{\Om}|\curlop v|^2 + \la|v|^2dx \geq \delta\int_{\Om}|\curlop v|^2dx
\qquad\text{for any }v\in \cV^+,
\end{equation}
and
\begin{equation}\label{CondInU1}
\int_{\Om}|\curlop v|^2 + \la|u|^2dx \leq 0
\qquad\text{for any }v\in \tcV.
\end{equation}
Moreover, if $\la_n<0$, i.~e.\ the kernel of the operator $(\curlop)^2+\la$ is trivial, then
\begin{equation}\label{CondInU1Ker}
\int_{\Om}|\curlop v|^2 + \la|u|^2dx \leq -\de\int_{\Om}|\curlop v|^2
\qquad\text{for any }v\in \tcV,
\end{equation}
provided $\de>0$ is small.

For any $v\in\V$ we denote $v^+\in\V^+$ and $\tv \in\tcV$ the corresponding summands such that $v=v^++\tv $.

Now we consider the functional $J:X=\cV\times\cW\to\R$ from Section~3 defined by
\[
\begin{aligned}
J(v,w)
 &= \frac12\int_{\Om}|\curlop v|^2\;dx + \frac\la2\int_{\Om} |v+\nabla w|^2\;dx
    - \int_{\Om}F(x,v+\nabla w)\;dx\\
 &= \frac12\|v\|_\cV^2 + \frac\la2\int_{\Om}(|v|^2+|\nabla w|^2)\;dx
    - \int_{\Om}F(x,v+\nabla w)\;dx.
\end{aligned}
\]
Setting
\[
X^+ := \{(v,0)\in\cV\times\cW:v\in\cV^+\} = \cV^+\times\{0\}\subset X
\]
and
\[
\tX := \{(v,w)\in\cV\times\cW:v\in\tcV\} = \tcV\times\cW\subset X
\]
we shall prove Theorem~\ref{thm:main} by showing that $J$ satisfies the assumptions (A1)-(A6) from Theorem~\ref{ThLink1}. Clearly $J$ has the form as in \eqref{EqJ} with
\[
I(v,w) = -\frac12\|\tv\|_\cV^2 - \frac\la2\int_\Om\left(|v|^2+|\nabla w|^2\right)
          + \int_\Om F(x,u+\nabla w).
\]

Observe that assumptions (F1)-(F3) imply that
for any $\eps>0$ there is a constant $c_{\eps}>0$ such that
\begin{equation}\label{eqest1}
|f(x,u)|\leq \eps|u| + c_{\eps}|u|^{p-1} \textnormal{ for any } x\in\Om,\;u\in\R^3
\end{equation}
and
\begin{equation}\label{eqest2}
\int_{\Om} F(x,u)\; dx \leq \eps|u|^2_{2} + c_{\eps}|u|^{p}_{p}
\textnormal{ for any } u\in L^p(\Om,\R^3).
\end{equation}

The next lemma shows that (A1)-(A2), (A4), (B1)-(B2) hold.

\begin{Lem}\label{LinkingLemma}
a) $I$ is of class $\cC^1$,  $I(v,w)\geq 0$ for any $(v,w)\in\V\times\W$, and $I$ is $\cT$-sequentially lower semicontinuous.

b) There is $r>0$ such that
$\displaystyle 0<\inf_{\stackrel{v\in \V^+}{\|v\|_\V=r}}J(v,0)$.

c) $\|v_n^+\|_\V + I(v_n+\nabla w_n)\to\infty$ as $\|(v,w)\|\to\infty$.

d) $I(t_n(v_n,w_n))/t_n^2\to\infty$ if $t_n\to\infty$ and $v_n^+\to v_0^+\neq 0$ as $n\to\infty$.
\end{Lem}

\begin{proof}
a)
Note that by (\ref{CondInU1})
$$
I(v,w) = -\frac12\left(\|\tv \|_\V^2 + \la |\tv |^2_{2}\right)
 - \frac\la2\int_\Om|v^+ + \nabla w|^2 + \int_\Om F(x,v+\nabla w) \geq 0.$$
Let $(v_n,w_n)\in\V\times\W$ be a sequence such that
$(v_n,w_n)\cTto (v_0,w_0)$ in $\V\times\W$. Since $\dim(\tcV)<\infty$ we may assume that $\tv_n\to \tu_0$ in $\V$. Since $F$ is convex in $u$ the map
$$
L^p(\Om,\R^3)\ni E\to\int_\Om F(x,E)\;dx\in\R
$$
is weakly sequentially lower semicontinuous and therefore
$$
\liminf_{n\to\infty}I(v_n,w_n)\geq I(v_0,w_0).
$$

b)
Note that by \eqref{NormInU0} and \eqref{eqest2} for any $u\in\V^+$
\begin{eqnarray*}
J(v,0)
 &=& \frac12\|v\|_\V^2 +\frac\la2|v|_{2}^2-\int_\Om F(x,v)\;dx
 \geq \frac\de2\|v\|_\V^2 -\eps|v|_2- c_\eps|v|_p^p\\
&\geq& \frac\de4\|v\|_\V^2 - C_1\|v\|_\V^p
\end{eqnarray*}
for some constant $C_1>0$. This implies b).

c) Suppose that $(\|v_n^+\|_\V)_{n}$ is bounded and $\|(v_n,w_n)\|\to\infty$ as $n\to\infty$. Since $\dim(\tcV)<\infty$ there holds
$|v_n+\nabla w_n|_p\to\infty$. Moreover by the orthogonality $\V^+\perp\tcV$,
$\V\perp\nabla\W$ in $L^2(\Om,\R^3)$ and by the H\"older inequality we have
\begin{equation}\label{EqOrthU_0}
\|\tv_n\|_\V^2\leq C_1|\tv_n|_2^2 \leq C_1|v_n|_{2}^2\leq C_1|v_n+\nabla w_n|_2^2
 \leq C_2|v_n+\nabla w_n|_p^2
\end{equation}
for some constants $C_2>C_1>0$. Now (F4) implies
\begin{eqnarray*}
I(v_n,w_n)
 &=& -\frac12\|\tv_n\|_\V^2 - \frac\la2|v_n+\nabla w_n|_2^2
      + \int_\Om F(x,n_n+\nabla w_n)\,dx\\\nonumber
 &\geq& -\frac12\|\tv_n\|_\V^2 - \frac\la2|v_n+\nabla w_n|_2^2
         + d|v_n+\nabla w_n|_p^p\\ \nonumber
 &\geq& -\frac{C_2}{2}|v_n+\nabla w_n|_p^2 + d|v_n+\nabla w_n|_p^p \to \infty
\end{eqnarray*}
because $|v_n+\nabla w_n|_p\to\infty$.

d) Note that
\begin{eqnarray*}
I(t_n(v_n,w_n))
 &=& -\frac12\|t_n\tv_n\|_\V^2 - \frac\la2|t_n v_n+t_n\nabla w_n|_2^2
      +\int_\Om F(x,t_nv_n+t_n\nabla w_n)\;dx\\
 &\geq& -\frac12t_n^2\|\tv_n\|_\V^2 - \frac\la2t_n^2|v_n+\nabla w_n|_2^2
         +d t_n^p|v_n+\nabla w_n|_p^p
\end{eqnarray*}
and then by \eqref{EqOrthU_0}
\begin{eqnarray*}
I(t_n(v_n,w_n))/t_n^2
 &\geq& -\frac12\|\tv_n\|_\V^2 - \frac\la2|v_n+\nabla w_n|_{2}^2
         +d t_n^{p-2}|v_n+\nabla w_n|_p^p\\
&\geq& -\frac{C_2}{2}|v_n+\nabla w_n|_p^2
        +d t_n^{p-2}|v_n+\nabla w_n|_p^p.
\end{eqnarray*}
If $\|(v_n,w_n)\|\to\infty$ then
$$
I(t_n(v_n,w_n))/t_n^2\to \infty.
$$
Suppose that $(\|(v_n,w_n)\|)_{n}$ is bounded. Then $(|v_n+\nabla w_n|_p)_n$ is bounded. If $|v_n+\nabla w_n|_p\to0$  then $|v_n+\nabla w_n|_2\to 0$ and by the orthogonality in $L^2(\Om,\R^3)$ we get $v_n\to 0$ in $L^2(\Om,\R^3)$ which contradicts $u_0\neq0$. Therefore $t_n^{p-2}|v_n+\nabla w_n|_p\to\infty$ as $n\to\infty$ and again
$$
I(t_n(v_n,w_n))/t_n^2\to \infty.
$$
\end{proof}

Let us consider the Nehari-Pankov manifold for $J$
\begin{eqnarray}\label{DefOfNehari}
\cN &:=& \{(v,w)\in(\V\times\W)\setminus (\tcV\times\W)|\;J'(v,w)[v,w]=0\\\nonumber
&&\hbox{ and }J'(v,w)[\phi,\psi]=0\hbox{ for any }(\phi,\psi)\in\tcV\times\W\}.
\end{eqnarray}

\begin{Lem}\label{LemB3check}
Condition (B3) holds.
\end{Lem}

\begin{proof}
Let $(v,w)\in \V\times \W$, $t\geq 0$, $\phi\in\tcV$, $\psi\in\W$ satisfy
$v+\nabla w \ne t(v+\nabla w)+\phi+\nabla\psi$. Observe that
\begin{eqnarray*}
&&
I'(v,w)\left[\frac{t^2-1}{2}(v,w)+t(\phi,\psi)\right]
  + I(v,w) - I(t(v,w)+(\phi,\psi))\\
&&\hspace{1cm}
 =\frac12\|\phi\|_\V^2 + \frac\la2|\phi+\nabla\psi|_2^2 + \int_\O\vp(t,x)\;dx\\
&&\hspace{1cm}
 =\frac12\big(\|\phi\|_\V^2+\la\|\phi\|_2^2\big) + \frac\la2|\nabla\psi|_2^2
     + \int_\Om\vp(t,x)\;dx
\end{eqnarray*}
where
$$
\vp(t,x)
 :=\langle f(x,v+\nabla w),\frac{t^2-1}{2}(v+\nabla w)+t(\phi+\nabla\psi)\rangle
    + F(x,v+\nabla w)-F(x,t(v+\nabla w)+\phi+\nabla\psi).
$$
Assume that $v(x)+\nabla w(x)\neq 0$. Note that by (F4) we have $\vp(0,x)< 0$ and
$$
\lim_{t\to\infty}\vp(t,x)=-\infty.
$$
Let $t_0\geq 0$ be such that $\vp(t_0,x)=\max_{t\geq 0}\vp(t,x)$. If $t_0=0$ then $\vp(t,x)<0$ for any $t\geq 0$. Let us assume that $t_0>0$. Then
 $\partial_t\vp(t_0,x)=0$, i.e.
$$
\langle f(x,v+\nabla w),t_0(v+\nabla w)+\phi+\nabla\psi\rangle
 - \langle f(x,t_0(v+\nabla w)+\phi+\nabla\psi),v+\nabla w\rangle = 0.
$$
If $\langle f(x,v+\nabla w),t_0(v+\nabla w)+\phi+\nabla\psi\rangle = 0$ then by (F4)
\begin{eqnarray*}
\vp(t_0,x)
 &=& \langle f(x,v+\nabla w),\frac{-t_0^2-1}{2}(v+\nabla w)\rangle
      + F(x,v+\nabla w)-F(x,t_0(v+\nabla w)+\phi+\nabla\psi)\\
 &<& -t_0^2F(x,v+\nabla w)-F(x,t_0(v+\nabla w)+\phi+\nabla\psi)\\
 &\leq& 0.
\end{eqnarray*}
If $\langle f(x,v+\nabla w),t_0(v+\nabla w)+\phi+\nabla\psi\rangle\neq 0$ then by (F7)
\begin{equation}\label{eq:phi}
\begin{aligned}
\vp(t_0,x)
 &= -\frac{(t_0-1)^2}{2}\langle f(x,v+\nabla w),v+\nabla w\rangle\\
 &\hspace{1cm}
     +t_0(\langle f(x,v+\nabla w),t_0(v+\nabla w)+\phi+\nabla\psi\rangle
     -\langle f(x,v+\nabla w),v+\nabla w\rangle)\\
&\hspace{1cm}+F(x,v+\nabla w)-F(x,t_0(v+\nabla w)+\phi+\nabla\psi)\\
&\leq-\frac{(\langle f(x,v+\nabla w),\phi+\nabla\psi\rangle)^2}{2\langle f(x,v+\nabla w),v+\nabla w\rangle}\\
 &\leq 0,
\end{aligned}
\end{equation}
and if $F(x,v+\nabla w)\neq F(x,t_0(v+\nabla w)+\phi+\nabla\psi)$ then $\vp(t_0,x)<0$. If $F(x,v+\nabla w)=F(x,t_0(v+\nabla w)+\phi+\nabla\psi)$ then (F7) yields
\[\langle f(x,v+\nabla w),t_0(v+\nabla w)+\phi+\nabla\psi\rangle
     \le \langle f(x,v+\nabla w),v+\nabla w\rangle.
\]
Therefore \eqref{eq:phi} implies
$$
\vp(t_0,x)\le-\frac{(t_0-1)^2}{2}\langle f(x,v+\nabla w),v+\nabla w\rangle.
$$
As a consequence, if $t_0\neq 1$ we deduce for $t\ge0$ that $\varphi(t,x)\leq\varphi(t_0,x)<0$. Now suppose $t_0=1$. If $\vp(t,x)=\vp(t_0,x)$ for some $0<t\neq t_0$ then $\partial_t\vp(t,x)=0$ and the above considerations imply $\vp(t,x)<0$. Summing up, we have shown that if $v(x)+\nabla w(x)\neq 0$ then $\vp(t,x)\leq 0$ for any $t\geq 0$ and $\vp(t,x)< 0$ if $t\neq 1$. Since $v+\nabla w\neq 0$ in $L^2(\Om,\R^3)$ we obtain from \eqref{CondInU1} for $t\geq 0$, $t\neq 1$:
\begin{equation}\label{EqIneq}
\frac12(\|\phi\|_\V^2 + \la|\phi|_2^2) + \frac\la2|\nabla h|_2^2 + \int_\Om\vp(t,x)\,dx
 < 0.
\end{equation}
Finally we need to consider the case $t=1$, hence $\phi+\nabla\psi\neq 0$. If $f$ is strictly convex then
$$
\vp(1,x)
 = \langle f(x,v+\nabla w),\phi+\nabla\psi\rangle + F(x,v+\nabla w)
    - F(x,v+\nabla w+\phi+\nabla\psi)
 < 0
$$
provided that $\phi(x)+\nabla\psi(x)\neq 0$. Thus if $\la=0$ or $\la=-\la_n$ then again \eqref{EqIneq} holds. If $\la\neq 0$ and $\la\neq-\la_n$  then \eqref{CondInU1Ker} holds and
$$
\frac12(\|\phi\|_\V^2 + \la|\phi|_2^2) + \frac\la2|\nabla\psi|_2^2 < 0.
$$
Therefore \eqref{EqIneq} is satisfied also for $t=1$.
\end{proof}

If we assume the Ambrosetti-Rabinowitz-type condition (F8), then we can show that Palais-Smale sequences are bounded, and taking into account the compact embedding of $\V$ into $L^p(\Om)$ we find a $\cT$-convergent subsequence. Without (F8) the situation is more complicated, therefore we consider Palais-Smale sequences on the Nehari manifold $\cN$. Namely, we show that $J$ satisfies $(PS)_c^{\cT}$ condition on $\cN$.

\begin{Lem}\label{lemmaconv2}
If $\{(v_n,w_n)\}\subset \cN$ is a $(PS)_c$-sequence for some $c>0$, i.e.\
$$
J(v_n,w_n)\to c\hbox{ and }J'(v_n,w_n)\to0,
$$
then, up to a subsequence,  $(v_n,w_n)\cTto (v_0,w_0)$ in $\V\times\W$ for some $(v_0,w_0)\in\V\times\W$.
\end{Lem}

\begin{proof}
Suppose that $(v_n,w_n)\in\cN$ and
\begin{equation}\label{eqCoercive1}
\|(v_n,w_n)\|\to\infty\qquad\text{as $n\to\infty$.}
\end{equation}
Let $\bar{v}_n(x) := \frac{v_n(x)}{\|(v_n,w_n)\|^2}$ and
$\bar{w}_n(x) := \frac{w_n(x)}{\|(u_n,w_n)\|^2}$ for $x\in\Om$. Observe that $\nabla\W$ is a closed subspace of $L^p(\Om,\R^3)$, and $\cl\V\cap\nabla \W=\{0\}$. Therefore there is a continuous projection of $\cl\V\oplus\nabla\W$ onto $\nabla\W$ and onto $\cl\V$ in $L^p(\Om,\R^3)$. Hence there is a constant $C_1>0$ such that
\begin{equation}\label{ineqPS1}
|\nabla w|_p\leq C_1|v+\nabla w|_{p}
\end{equation}
and
\begin{equation}\label{ineqPS1_u}
|v|_p\leq C_1|v+\nabla w|_{p}.
\end{equation}
for any $v\in\V$ and $w\in\W$. Note that by (F4) we obtain for almost all $n\in\N$
\begin{eqnarray*}
\|v_n\|_\V^2&\geq& 2J(v_n,w_n) - \la(|v_n|_{2}^2+|\nabla w_n|_{2}^2) + 2d|v_n+\nabla w_n|_{p}^p\\
&\geq& c+2d|v_n+\nabla w_n|_{p}^p,
\end{eqnarray*}
and by (\ref{ineqPS1}) we get
\begin{equation*}
2\|v_n\|_\V^2\geq \|v_n\|_\V^2 + c + 2d|v_n+\nabla w_n|_p^p
 \geq \|v_n\|^2_\V + c + 2dC_1|\nabla w_n|_p^p.
\end{equation*}
If $|\nabla w_n|_p\to 0$ then
$$
\|v_n\|_\V^2 + c + 2dC_1|\nabla w_n|_p^p \geq \|(v_n,w_n)\|^2
$$
for almost all $n\in \N$. Otherwise, passing to a subsequence, we have $\liminf_{n\to\infty}|\nabla w_n|_p > 0$ and
$$
\|v_n\|_\V^2 + c + 2dC_1|\nabla w_n|_{p}^p \geq C_2\|(v_n,w_n)\|^2
$$
for some constant $C_2>0$. Therefore we have shown that, up to a subsequence,
$$
2\|v_n\|_\V^2\geq \min\{1,C_2\} \|(v_n,w_n)\|^2
$$
and thus $(\|\bar{v}_n\|_\V)_n$ is bounded away from $0$. Therefore we may assume that
$$
\bar{v}_n = \bar{v}_n^+ + \widetilde{\bar{v}}_n,
$$
where $\bar{v}_n^+ \in \V^+$, $\widetilde{\bar{v}}_n \in \tcV$,
\begin{eqnarray*}
&&\bar{v}_n^+ \to \bar{v}^+_0 \hbox{ in }L^p(\Om,\R^3),\\
&&
\bar{v}_n\weakto \bar{v}_0 = \bar{v}^+_0+\widetilde{\bar{v}}_0 \hbox{ in }\V, \\
&&\bar{v}_n(x)\to \bar{v}_0(x)\hbox{ a.e.\ on }\Om
\end{eqnarray*}
and
$\bar{v}_0 \neq 0$. Indeed, if $\bar{v}_0 = 0$ then $\bar{v}^+_0 = \widetilde{\bar{v}}_0=0$. Observe that by Lemma \ref{LinkingLemma}, Lemma \ref{LemB3check} and Proposition \ref{prop:geometry}, condition (A5) is satisfied. Hence
$$
J(v_n,w_n)
 \geq J(t\bar{v}_n^+,0)
  = \frac{t^2}{2}\|\bar{v}_n^+\|^2_\V - \int_\Om F(x,t\bar{v}_n^+)\,dx
$$
for any $t\geq 0$. Therefore by (\ref{eqest2})
$$
c \geq \frac{t^2}{2}\liminf_{n\to\infty}\|\bar{v}_n^+\|^2_\V
  =\frac{t^2}{2}\liminf_{n\to\infty}\|\bar{v}_n\|^2_\V
$$
for any $t\geq 0$. Now we obtain a contradiction because $(\|\bar{v}_n\|_\V)_n$ is bounded away from $0$. Therefore $\bar{u}_0\neq 0$. Observe that (F4) and \eqref{ineqPS1_u} imply
\begin{eqnarray*}
\frac{J(v_n,w_n)}{\|(v_n,w_n)\|^2}
 &\leq& \frac12\|\bar{v}_n\|^2_\V + \frac\la2|\bar{v}_n + \nabla\bar{w}_n|^2_2
         -d\int_\Om|v_n+\nabla w_n|^p\,dx \\
 &\leq& \frac12\|\bar{v}_n\|^2_\V + \frac\la2|\bar{v}_n+\nabla\bar{w}_n|^2_2
         -dC_1^{-p}\int_\Om|v_n|^p\,dx \\
 &=& \frac12\|\bar{v}_n\|^2_\V + \frac\la2|\bar{v}_n+\nabla\bar{w}_n|^2_2
      -dC_1^{-p}\int_\Om|v_n|^{p-2}|\bar{v}_n|^2\,dx.
\end{eqnarray*}
Since $v_n(x)=\bar{v}_n(x)\|(v_n,w_n)\|^2\to\infty$ if $\bar{v}_0(x)\neq 0$, then by the Fatou's lemma
$$
\int_\Om|v_n|^{p-2}|\bar{v}_n|^2\,dx \to \infty
$$
and we obtain a contradiction with $\frac{J(v_n,w_n)}{\|(v_n,w_n)\|^2}\to 0$ as $n\to\infty$. Therefore $\|(v_n,w_n)\|$ is bounded and we may assume, up to a subsequence,
$$
v_n \weakto v_0\hbox{ in }\V,\;v_n \to v_0 \hbox{ in }L^p(\Om,\R^3),
 \hbox{ and }w_n\weakto w_0\hbox{ in }\W
$$
for some $(v_0,w_0)\in\V\times\W$. Note that
\begin{eqnarray*}
&&J'(v_n,w_n)[v_n-v_0,0]\\
&=&\hspace{1cm}
 \|v_n-v_0\|_\V^2+\int_\Om\langle \nabla v_0,\nabla v_n-\nabla v_0\rangle\,dx
      +\la\int_\Om\langle v_n+\nabla w_n,v_n-v_0\rangle \,dx\\
&&\hspace{2cm}
     -\int_\Om\langle f(x,v_n+\nabla w_n),v_n-v_0\rangle \,dx.
\end{eqnarray*}
Since $(v_n)_{n}$ is bounded in $\V$, $v_n\to v_0$ in $L^2(\Om,\R^3)$ and
$(f(x,v_n+\nabla w_n))_{n}$ is bounded in $L^{\frac{p}{p-1}}(\Om,\R^3)$ we deduce
$\|v_n-v_0\|_\V\to 0$.
\end{proof}

\begin{Prop}\label{PropCritical}
There is a critical point $(v_0,w_0)\in\cN$ of $J$ such that
$$
J(v_0,w_0) = \inf_\cN\J.
$$
\end{Prop}

\begin{proof}
In view of Lemma \ref{LinkingLemma}, Lemma \ref{LemB3check} and Proposition \ref{prop:geometry}, we know that (A1)-(A2) and (A4)-(A6) are satisfied. From Lemma \ref{lemmaconv2} we obtain that $J$ satisfies the $(PS)_c^{\cT}$-condition on $\cN$ for any $c>0$. In order to apply Theorem \ref{ThLink1} it remains to show that (A2) holds. Assume that $v_n^+\to v^+_0$, $v_n'\weakto u_0'$ in $\V$, $w_n\weakto w_0$ in $\W$ and $I(v_n+\nabla w_n)\to I(v_0+\nabla w_0)$. Since along a subsequence $\tv_n\to\tv_0$ in $\V$, then
$$
-\frac\la2\int_\Om|v^+_n+\nabla w_n|^2\,dx + \int_\Om F(x,v_n+\nabla w_n)\; dx
 \to -\frac\la2\int_\Om|v^+_0+\nabla w_0|^2\,dx + \int_\Om F(x,v_0+\nabla w_0)\,dx.
$$
If $\la<0$ then by the weakly sequentially lower semicontinuity
\begin{equation}\label{EqconvinL2}
|v^0_n+\nabla w_n|_2 \to |v^+_0+\nabla w_0|_2
\end{equation}
and since $v^+_n+\nabla w_n \weakto v^+_0+\nabla w_0$ in $L^p(\Om,\R^3)$ then, up to a subsequence, $v^+_n+\nabla w_n \weakto v^+_0+\nabla w_0$ in $L^2(\Om,\R^3)$ and by
(\ref{EqconvinL2}) we have $v_n+\nabla w_n\to v_0+\nabla w_0$ in $L^2(\Om,\R^3)$. Hence
$$
v_n+\nabla w_n\to v_0+\nabla w_0\hbox{ a.e.\ on }\Om.
$$
If $\la=0$ then by (F6) and for any $0 < r \leq R$
\begin{equation}\label{eqLambda0case}
m := \inf_{\genfrac{}{}{0pt}{}{x\in\Om,u_1,u_2\in\R^3}{r\leq|u_1-u_2|,|u_1|,|u_2|\leq R}}\;
       \frac12(F(x,u_1)+F(x,u_2)) - F\left(x,\frac{u_1+u_2}{2}\right) > 0.
\end{equation}
Observe that by the convexity of $F$ in $u$
$$
0
 \leq \limsup_{n\to\infty}\int_{\R^3}\frac12(F(x,v_n+\nabla w_n) + F(x,v_0+\nabla w_0))
        - F\left(x,\frac{v_n+\nabla w_n+v_0+\nabla w_0}{2}\right)\,dx
 \leq 0.
$$
Therefore setting
$$
\Om_n:=\{x\in\Om\;|v_n+\nabla w_n-(v_0+\nabla w_0)|\geq r,\;|v_n+\nabla w_n|\leq R,\;
          |v_0+\nabla w_0|\leq R\}
$$
there holds
$$
\mu(\Om_n)m
 \leq \int_{\R^3}\frac12(F(x,v_n+\nabla w_n) + F(x,v_0+\nabla w_0))
       - F\left(x,\frac{v_n+\nabla w_n+v_0+\nabla w_0}{2}\right)\, dx
$$
and thus $\mu(\Om_n)\to 0$ as $n\to\infty$. Since $0<r\leq R$ are arbitrary chosen, we deduce
$$
v_n+\nabla w_n\to v_0+\nabla w_0\hbox{ a.e.\ on }\Om.
$$
Finally observe that
\begin{eqnarray*}
&&\int_\Om F(x,v_n+\nabla w_n) - F(x,v_n+\nabla w_n-(v_0+\nabla w_0))\,dx\\
&&\hspace{1cm}
   = \int_\Om\int_0^1\frac{d}{dt}F(x,v_n+\nabla w_n+(t-1)(v_0+\nabla w_0))\,dtdx\\
&&\hspace{1cm}
  = \int_0^1\int_\Om\langle f(x,v_n+\nabla w_n+(t-1)(v_0+\nabla w_0)),
          v_0+\nabla w_0\rangle\,dxdt.
\end{eqnarray*}
Since $f(x,v_n+\nabla w_n+(t-1)(v_0+\nabla w_0))\to f(x,t(v_0+\nabla w_0))$ a.e.\ on $\Om$ then in view of the Vitali convergence theorem
\begin{eqnarray*}
&&\int_\Om F(x,v_n+\nabla w_n)-F(x,v_n+\nabla w_n-(v_0+\nabla w_0))\,dx\\
&&\hspace{1cm}
  \to \int_0^1\int_\Om\langle f(x,t(v_0+\nabla w_0)),v_0+\nabla w_0\rangle\,dxdt
  = \int_\Om F(x,v_0+\nabla w_0)\,dx
\end{eqnarray*}
as $n\to\infty$. Moreovoer, since
$\int_\Om F(x,v_n+\nabla w_n)\to \int_\Om F(x,v_0+\nabla w_0)\,dx$
there holds
\begin{equation}\label{EqConvf}
\int_\Om F(x,v_n+\nabla w_n-(v_0+\nabla w_0))\,dx \to 0
\end{equation}
and by (F4)
$$
|v_n+\nabla w_n-(v_0+\nabla w_0)|_{p} \to 0.
$$
Therefore $\nabla w_n\to\nabla w_0$ in $L^p(\Om,\R^3)$.
\end{proof}

\begin{altproof}{Theorem~\ref{thm:main}}
By Proposition \ref{PropCritical} there exists a critical point $(v_0,w_0)$ of $J$ such that $(v_0,w_0)\in\cN$, in particular $v_0\neq0$. Proposition \ref{PropSolutE} yields that
$E = v_0+\nabla w_0$ is a critical point of $\fJ$ on the space $\V\oplus\nabla \W$ and hence a solution of \eqref{eq:problem}. Moreover $E\neq 0$ since $v_0$ and $\nabla w_0$ are orthogonal in $L^2(\Om,\R^3)$ and $E$ is a ground state solution, i.e.\
$$
\fJ(A)=\inf_{(v,w)\in\cN} \fJ(v+\nabla w).
$$
\end{altproof}

\section{Proof of Theorem~\ref{thm:sym1}{}}\label{sec:proof-sym}

The symmetry condition (S) allows to adapt the approach from \cite{BenForAzzAprile} to our setting. Since $\Om$ is invariant under $G = O(2)\times{1}\subset O(3)$ we can define an action of $g \in G$ on $v \in \cV$ and on $w \in \cW$ as follows:
\[
(g*v)(x) := g\cdot v(g^{-1}x) \qquad\text{and}\qquad
(g*w)(x) := w(g^{-1}x).
\]
It is not difficult to check that this defines an isometric linear (left) action of $G$ on $X = \cV\times\cW$. In particular there holds $gv\in\cV$ and
$\int_\Om|\curlop (g*v)|^2dx = \int_\Om|\curlop v|^2dx$. Similarly, $g*w\in\cW$, and
$|\nabla (g*w)|_p = |\nabla w|_p$. Moreover, as a consequence of (S), $J$ is invariant with respect to this action: $J(g*v,g*w) = J(v,w)$. Let $X^G = \cV^G\times\cW^G$ be the fixed point set of this action, so $\cV^G$ consists of all $G$-equivariant vector fields $v\in\cV$, and $\cW^G$ consists of all $G$-invariant functions $w\in\cW$. By the principle of symmetric criticality, a critical point of the constrained functional $J|_{X^G}$ is a critical point of $J$.

Next we decompose any $v \in \cV^G$ as $v = v_\rho+v_\tau+v_\zeta$ with:
\[
v_\rho(x)
 = \be(r,x_3)\begin{pmatrix}x_1\\x_2\\0\end{pmatrix},\quad
v_\tau(x)
 = \al(r,x_3)\begin{pmatrix}-x_2\\x_1\\0\end{pmatrix},\quad
v_\zeta(x)
 = \ga(r,x_3)\begin{pmatrix}0\\0\\1\end{pmatrix},
\]
where $r=\sqrt{x_1^2+x_2^2}$. That the coefficient functions $\al,\be,\ga$ depend only on $(r,x_3)$ is an immediate consequence of the $G$-equivariance of $v$, i.~e.\ $v(gx)=g\cdot v(x)$. As in \cite[Lemma~1]{BenForAzzAprile} one sees that
$\nabla v_\rho,\nabla v_\tau,\nabla v_\zeta\in L^2(\Om;\R^3)$.
Clearly $\div v_\tau = 0$, hence $v_\tau, v_\rho+v_\zeta \in \cV$. Therefore the map
\[
S:\cV^G \to \cV^G, \quad S(v_\rho+v_\tau+v_\zeta) := -v_\rho+v_\tau-v_\zeta
\]
is well defined. A direct computation shows that
\begin{equation}\label{eq:S-isom1}
\langle\curlop v_\tau(x),\curlop v_\rho(x)\rangle = 0 =
 \langle\curlop v_\tau(x),\curlop v_\zeta(x)\rangle
\end{equation}
and
\begin{equation}\label{eq:S-isom2}
\langle\nabla v_\tau(x),\nabla v_\rho(x)\rangle = 0 =
 \langle\nabla v_\tau(x),\nabla v_\zeta(x)\rangle
\end{equation}
\eqref{eq:S-isom1} implies that $S$ is a linear isometry, and so is
\[
T:X^G=\cV^G\times\cW^G\to X^G,\quad T(v,w) := (Sv,-w).
\]
Clearly we have $T^2=\id$, and
\[
(X^G)^T:=\{(v,0)\in\cV^G\times\cW: Sv=v\} = \{(v,0)\in\cV^G\times\cW: v=v_\tau\}.
\]
As a consequence of \eqref{eq:S-isom1}, \eqref{eq:S-isom2}, and hypothesis (S), $J$ is invariant under this action:
$$
J(Tu)=J(u) \quad\text{for all $u=(v,w)\in X^G$.}
$$
Applying the principle of symmetric criticality once more we see that it suffices to find critical points of $J|_{(X^G)^T}$.

The above discussion shows that we only need to find critical points of the functional
\[
J_Y: Y:=\{v\in\cV: Sv=v\}\to\R
\]
defined by
\[
\begin{aligned}
J_Y(v)
 &= J(v,0)
  = \frac12\int_{\Om}|\curlop v|^2\;dx + \frac\la2\int_{\Om} |v|^2\;dx
    - \int_{\Om}F(x,v)\;dx\\
 &= \frac12\|v\|_\cV^2 + \frac\la2\int_{\Om} |v|^2\;dx
    - \int_{\Om}F(x,v)\;dx.
\end{aligned}
\]
Here we can apply standard critical point theory. Since $F$ is even as a consequence of (S), the existence of an unbounded sequence of solutions follows from the fountain theorem in \cite{Bartsch:1993}, see also \cite{Willem}. Details are left to the reader.

{\bf Acknowledgements.} T.~B.\ would like to thank Wolfgang Reichel (Karlsruhe) for proposing the problem and for invaluable help. He would also like to thank Michael Plum (Karlsruhe) and Tobias Weth (Frankfurt) for various discussions on the topic. J.~M. would like to thank the members of the Department of Mathematics at the University of Giessen, where part of this work has been done, for their invitation and hospitality.

{\sc Address of the authors:}\\[1em]
\parbox{8cm}{Thomas Bartsch\\
 Mathematisches Institut\\
 Universit\"at Giessen\\
 Arndtstr.\ 2\\
 35392 Giessen\\
 Germany\\
 Thomas.Bartsch@math.uni-giessen.de}
\parbox{7cm}{Jaros\l aw Mederski\\
 Nicolaus Copernicus University \\
 ul.\ Chopina 12/18\\
 87-100 Toru\'n\\
 Poland\\
 jmederski@mat.umk.pl\\
 }


\begin{thebibliography}{99}
\baselineskip 2 mm

\bibitem{Amrouche}  C. Amrouche, C. Bernardi, M. Dauge, V. Girault:
{\em Vector potentials in three-dimensional non-smooth domains},
Math. Methods Appl. Sci. {\bf 21} (1998), no. 9, 823--864.

\bibitem{BenForAzzAprile} A. Azzollini, V. Benci, T. D'Aprile, D. Fortunato:
{\em Existence of Static Solutions of the Semilinear Maxwell Equations},
Ric. Mat. {\bf 55} (2006), no. 2, 283--297.

\bibitem{Bartsch:1993} T. Bartsch:
{\em Infinitely many solutions of a symmetric Dirichlet problem}, Nonlin. Anal. {\bf 20} (1993), no. 10, 1205--1216.

\bibitem{BartschDing} T. Bartsch, Y. Ding:
{\em Deformation theorems on non-metrizable vector spaces and applications to critical point theory},
Mathematische Nachrichten {\bf 279} (2006), no. 12, 1267--1288.

\bibitem{BenFor} V. Benci, D. Fortunato:
{\em Towards a unified field theory for classical electrodynamics},
Arch. Rational Mech. Anal. {\bf 173} (2004), 379--414.

\bibitem{BenciRabinowitz} V. Benci, P. H. Rabinowitz:
{\em Critical point theorems for indefinite functionals},
Invent. Math. {\bf 52} (1979), no. 3, 241--273.

\bibitem{DAprileSiciliano} T. D'Aprile, G. Siciliano:
{\em Magnetostatic solutions for a semilinear perturbation of the Maxwell equations},
Adv. Differential Equations {\bf 16} (2011), no. 5--6, 435-466.

\bibitem{Corvellec-Degiovanni-Marzocchi:1993}
J.-N. Corvellec, M. Degiovanni, and M. Marzocchi:
{\em Deformation properties for continuous functionals and critical point
theory}, Topol. Methods Nonlin. Anal. {\bf 1} (1993), no. 1, 151--171.

\bibitem{DingBook} Y. Ding:
{\em Variational Methods for Strongly Indefinite Problems},
Interdisciplinary Mathematical Sciences {\bf 7}. World Scientific Publishing 2007.

\bibitem{Doerfler} W. D\"orfler, A. Lechleiter, M. Plum, G. Schneider, C. Wieners:
{\em Photonic Crystals: Mathematical Analysis and Numerical Approximation},
Springer Basel 2012.

\bibitem{KryszSzulkin} W. Kryszewski, A. Szulkin:
{\em Generalized linking theorem with an application to semilinear Schr\"odinger equation}, Adv. Diff. Eq. 3 (1998), 441--472.

\bibitem{Monk} P. Monk:
{\em Finite Element Methods for Maxwell's Equations},
Oxford University Press 2003.

\bibitem{Pankov} A. Pankov:
{\em Periodic Nonlinear Schr\"odinger Equation with Application to Photonic Crystals},
Milan J. Math. {\bf 73} (2005), 259--287.

\bibitem{Rabinowitz:1986} P. Rabinowitz:
{\em Minimax Methods in Critical Point Theory with Applications to Differential Equations},
CBMS Regional Conference Series in Mathematics, Vol. {\bf 65}, Amer. Math. Soc., Providence, Rhode Island 1986.

\bibitem{Struwe} M. Struwe:
{\em Variational Methods}, Springer-Verlag 2008.

\bibitem{Stuart91} C. A. Stuart:
{\em Self-trapping of an electromagnetic field and bifurcation from the essential spectrum}, Arch. Rational Mech. Anal. {\bf 113} (1991), no. 1, 65--96.

\bibitem{Stuart04} C. A. Stuart:
{\em Modelling axi-symmetric travelling waves in a dielectric with nonlinear refractive index},
Milan J. Math. {\bf 72} (2004), 107--128.

\bibitem{StuartZhou96}
C.A. Stuart, H.S. Zhou:
{\em A variational problem related to self-trapping of an electromagnetic field},
Math. Methods Appl. Sci. {\bf 19} (1996), no. 17, 1397--1407.

\bibitem{StuartZhou01} C.A. Stuart, H.S. Zhou:
{\em Existence of guided cylindrical TM-modes in a homogeneous self-focusing dielectric}, Ann. Inst. H. Poincar\'e Anal. Non Lin\'eaire {\bf 18} (2001), no. 1, 69--96.

\bibitem{StuartZhou03} C.A. Stuart, H.S. Zhou:
{\em A constrained minimization problem and its application to guided cylindrical TM-modes in an anisotropic self-focusing dielectric},
Calc. Var. Partial Differential Equations {\bf 16} (2003), no. 4, 335--373.

\bibitem{StuartZhou05} C.A. Stuart, H.S. Zhou:
{\em Axisymmetric TE-modes in a self-focusing dielectric}, SIAM J. Math. Anal. {\bf 37} (2005), no. 1, 218--237.

\bibitem{StuartZhou10} C.A. Stuart, H.S. Zhou:
{\em Existence of guided cylindrical TM-modes in an inhomogeneous self-focusing dielectric}, Math. Models Methods Appl. Sci. {\bf 20} (2010), no. 9, 1681--1719.

\bibitem{SzulkinWethHandbook} A. Szulkin, T. Weth:
{\em The method of Nehari manifold.}
Handbook of nonconvex analysis and applications, 597--632, Int. Press, Somerville, 2010.

\bibitem{SzulkinWeth} A. Szulkin, T. Weth:
{\em Ground state solutions for some indefinite variational problems},
J. Funct. Anal. {\bf 257} (2009), no. 12, 3802--3822.

\bibitem{Willem} M. Willem:
{\em Minimax Theorems}, Birkh\"auser Verlag 1996.

\end{thebibliography}
\end{document}